\documentclass [12pt]{article}
\usepackage{amsmath,amsthm}
\usepackage{amssymb,latexsym}
\usepackage{graphicx}
\usepackage{setspace}
\usepackage[mathscr]{eucal}
\usepackage{subfig}
\usepackage[pdftex,bookmarks,bookmarksnumbered=true,colorlinks=true,linkcolor=black,citecolor=black]{hyperref}
\usepackage[citation-order]{amsrefs}
\numberwithin{equation}{section}
\newtheorem{lemma}{Lemma}
\newtheorem{theorem}{Theorem}

\begin{document}
\title{Periodic and Chaotic Traveling Wave Patterns in Reaction--Diffusion/ Predator--Prey Models \\with General Nonlinearities}
\author{Stefan C. Mancas\footnote{Department of Mathematics, Embry--Riddle Aeronautical University, Daytona Beach, FL 32114--3900, stefan.mancas@erau.edu}\\and\\Roy S. Choudhury\footnote{Department of Mathematics, University of Central Florida, Orlando, FL. 32816--1364, choudhur@longwood.cs.ucf.edu }}
\date{}
\maketitle
\begin{abstract}

Traveling wavetrains in generalized two--species predator--prey models and two--component reaction--diffusion equations are considered.  The stability of the fixed points of the traveling wave ODEs (in the usual "spatial'' variable) is considered. For general functional forms of the nonlinear prey birthrate/prey deathrate or reaction terms, a Hopf bifurcation is shown to occur at two different critical values of the traveling wave speed. The post--bifurcation dynamics is investigated for five different functional forms of the nonlinearities. In cases where the bifurcation is supercritical, the post--bifurcation behaviour yields stable periodic orbits of the traveling--wave ODEs in the spatial variable. These correspond to stable periodic wavetrains of the full PDEs. Subcritical Hopf bifurcations yield more complex post--bifurcation dynamics in the PDE wavetrains. In special cases where the subcritical bifurcation marks the end of the regime of stability, the post--bifurcation behavior in the spatial ODEs is chaotic, corresponding to wavetrains of the original PDEs which are spatially coherent, but have chaotic temporal dynamics. All the models are integrated numerically to investigate the post--bifurcation dynamics and chaotic regimes are characterized by computing power spectra, autocorrelation functions, and fractal dimensions.

\end{abstract}

\section{Introduction}

Morphogenesis or the occurrence  of spatial form and pattern evolving from a spatially homogenous state is a fundamental problem in developmental biology. A seminal contribution to this problem was made by Turing \cite {Turing} who studied reaction--diffusion equations of the form
\begin{align} \label{system}
\frac{\delta N}{\delta t}&=R_1(N,P)+D_1 \frac{\delta^2 N}{\delta x^2}\\
\frac{\delta P}{\delta t}&=R_2(N,P)+D_2 \frac{\delta^2 P}{\delta x^2} \notag
\end{align}
In \cite {Turing}, the reaction functions (or kinematic terms) $R_1$ and $R_2$ were polynomials. However, the fundamental, and somewhat surprising, result that diffusion could destabilize an otherwise stable equilibrium leading to nonuniform spatial patterns  (referred to as prepattern) is not dependent on particular forms of $R_1$ and $R_2$.

The Turing instability in reaction--diffusion  models thus provided a plausible and robust mechanism for the establishment  of spatial prepattern, which could then generate biological patterns for gene activation. Numerous extensions  and applications followed. These include early theoretical and analytical extensions \cites{Turing, Wardlaw, Othmer}. In particular, Segel and Jackson \cite{Segel:1} showed that spatial patterns may occur via Turing instability in macroscopic (extended Lotka--Volterra) models in population biology as well, particularly for species dispersing at different rates. They also provided a lucid physical explanation of how diffusion could indeed generate instability, contrary to its usual interpretation as a smoothing mechanism. Applications in development  biology were stimulated by the work of Meinhards and Gierer \cites{Gierer:1, Gierer:2, Meinhardt}, primarily consisting of numerical simulations of reaction--diffusion systems in various geometries. Analytical work has confirmed and extended the results of \cites{Gierer:1, Gierer:2, Meinhardt}, including bifurcation analysis and investigations of nonstationary (traveling--wave) patterns, spirals, solitary peaks, and fronts \cites{Granero,Keener,Segel:2,Smoller,Rothe}. These are reviewed in \cite{Fife}. Other work has focused  on explaining the properties of spatial patterns \cites{Murray:1,Mimura,Bard,Schaller} on the basis of chemical interactions and geometric considerations. Alternative explanations of pattern--formation, not based on reaction--diffusion equations and the Turing mechanism, have also been investigated \cite{Murray:2}. Recent reviews of these  and other related work on spatial pattern formation are given by Levin and Segel \cite{Levin}, Murray \cite{Murray:3} and Edelstein--Keshet \cite{Keshet}.

In order to incorporate various realistic physical effects which may cause at least one of the physical variables to depend on the past history of the system, it is often necessary to introduce time--delays into the governing equations. Factors that introduce time lags may include age structure of the population (influencing the birth and death rates), maturation periods (thresholds), feeding times and hunger coefficients in predator--prey interactions, reaction times, food storage times, and resource generation times.  Models incorporating time delays in diverse spatially--homogenous biological systems are extensively reviewed by MacDonald \cite{MacDonald}, and in the context of predator--prey models, by Cushing \cite{Cushing}.  These include continuous models such as the Kolmogorov, May, Holling, Hsu, Leslie, and Caperon models, as well as discrete models.

Consider \eqref{system} for the general two--species predator--prey model \cite{Choudhury:1} with
\begin{align} \label{r}
R_1(N,P)&=N F(N)-\alpha NP-\frac{\tilde{\epsilon}N^2}{k}\\
R_2(N,P)&=-PG(P)+\beta NP \notag,
\end{align}
where $N(t)$ and $P(t)$ are the prey and predator populations, respectively, $\tilde{\epsilon}$ is the birth rate of the prey, $k>0$ is the carrying capacity, $\alpha$ is  the rate of predation per predator, and $\beta$ is the rate of the prey's contribution to predator growth. 

In this paper, we initiate a fresh and detailed investigation of traveling spatial wave patterns of \eqref{system}. In particular, we shall investigate in detail wavetrains with periodic and chaotic spatial variation. Toward this end, we consider traveling wave solutions of \eqref{system} in the form of 
\begin{align}\label{ansatz}
N(x,t)&= N(\zeta)\\
P(x,t)&= P(\zeta), \notag
\end{align}
where $\zeta=x-vt$ is the traveling wave, or "spatial", variable, and $v$ is the translation or wave speed, which will act as our bifurcation parameter. Substitution of Eqns. \eqref{r},\eqref{ansatz} in \eqref{system} leads, after some simplification, to the four--mode dynamical system
\begin{align}\label{4D}
\dot{N}&=M\\
\dot{M}&=\frac{1}{D_1}\Big(-vM-NF(N)+\alpha NP+ \frac{\tilde{\epsilon}N^2}{k} \Big)\notag \\
\dot{P}&=Q \notag \\
\dot{Q}&=\frac{1}{D_2}\Big(-vQ+PG(P)-\beta NP)\Big), \notag 
\end{align}
where the overdot denotes $\frac{d}{d \zeta}$. 

Here, however, we will follow [24,28] to consider the stability of the equilibria and the Hopf bifurcations of \eqref{4D} for general functions $F(N)$ and $G(P)$.  This is done in sections \S2 and \S3. In section \S4 we consider \eqref{4D} for specific choices of $F(N)$ and $G(P)$ to determine the regions of phase--space where the system is volume contracting (dissipative), or volume expanding (dilatory). Also, note that the function $F(N)$ incorporates the prey birth rate, and similarly, the function $G(P)$ incorporates the predator death rate  and is chosen so that this rate increases with predator density P. Section \S5 considers the stability of  physically relevant equilibria and Hopf bifurcation points for specific parameter values and choices of $F(N)$, and $G(P)$. Possible chaotic regimes are also delineated there.  The systems are numerically integrated and chaotic regimes are characterized by computing power spectra, correlation function and fractal dimensions \cite{Rothe}. Section \S6 summarizes the results and presents the conclusions.

\section{Linear stability analysis}

The equilibrium, critical or fixed points of the system \eqref{4D} (only nontrivial points are relevant since both the predator or prey population can not be zero) are
\begin{equation}
\label{fixedpoints}
(N_0,M_0,P_0,Q_0)=\Big(\frac{G(P_0)}{\beta}, 0, \frac{F(N_0)-\frac{\tilde{\epsilon}N_0}{k}}{\alpha},0\Big)
\end{equation}
In this section we will consider the Turing bifurcations in general predator--prey systems by considering the system \eqref{4D} for general $F(N)$ and $G(P)$, which incorporate the  prey birth rate and predator death rate. For numerical purposes in, the functions $F(N)$ and $G(P)$ are subsequently chosen to be
\begin{itemize}
\item[A.] $F(N)=\epsilon$, $G(P)=\gamma$, $\tilde{\epsilon}=0$,
\item[B.] $F(N)=\epsilon$, $G(P)=\gamma$, $\tilde{\epsilon}=\epsilon$,
\item[C.] $F(N)=k_0$, $G(P)=\bar{d}+\bar{c} P$, $\tilde{\epsilon}=\epsilon$,
\item[D.] $F(N)=k_0(1+\frac{N}{k})$, $G(P)=d+c P$, $\tilde{\epsilon}=\epsilon$,
\item[E.] $F(N)=\frac{1+\delta N}{1+N^2}$, $G(P)=\gamma(1+k P^2)$, $\tilde{\epsilon}=0$.
\end {itemize}
For the remainder of this paper we shall refer to these cases as System A--E. System B is a modified Lotka--Volterra two species model with diffusion, and $\gamma$ being the death rate of the predator. Notice that qualitative features of such models have been considered earlier, for instance for the Kolmogorov model without delay \cite{Wardlaw} and the May model with delay \cite{Turing}. 

Following standard methods of phase--plane analysis the Jacobian matrix of \eqref{4D} evaluated at the fixed point $(N_0,M_0,P_0,Q_0)$ is
\begin{eqnarray}
J=
\left(\begin{array}{cccc}
0 & 1 & 0& 0\\
\frac{\alpha P_0+\frac{2 \tilde{\epsilon} N_0}{k}-F(N_0)-N_0 F'(N_0)}{D_1} & -\frac{v}{D_1} & \frac{\alpha N_0}{D_1} &0\\
0 & 0 & 0 &1\\
-\frac{\beta P_0}{D_2} & 0 & \frac{-\beta N_0+G(P_0)+P_0G'(P_0)}{D_2} &-\frac{v}{D_2}\\
\end{array}\right).\label{jaco}
\end{eqnarray}
The eigenvalues of this matrix satisfy the characteristic equation
\begin{equation} \label{charac}
g(\lambda)=\lambda^4+b_1\lambda^3+b_2\lambda^2+b_3\lambda +b_4=0
\end{equation}
where
$b_i$ with $i=1,...,4$ are given by
\begin{align}\label{b}
b_1&=\frac{(D_1+D_2)v}{D_1D_2}\\
b_2&=\frac{\alpha(\beta k v^2-\epsilon D_2G_0+kD_2 F'_0G_0)+D_1(\epsilon G_0- \beta k F)G'_0}{\alpha \beta k D_1D_2}\notag \\
b_3&=\frac{v\big( -\epsilon \alpha G_0+\alpha k F'_0 G_0 +(\epsilon G_0-\beta k F)G'_0\big)}{\alpha \beta k D_1D_2}\notag \\
b_4&=\frac{(\beta k F_0-\epsilon G_0)\big( \alpha \beta k +(\epsilon-k F_0)G'_0\big)G_0}{\alpha \beta^2 k^2 D_1D_2} \notag 
\end{align}
where $F_0=F(N_0)$, $G_0=G(P_0)$, $F'_0=\frac{d F(N)}{d N}  \big|N_0$, and $G'_0=\frac{d G(P)}{d P}  \big|P_0$.
The Routh--Hurwitz criteria \cite{Murray:3}, giving the necessary and sufficient conditions $Re(\lambda_i)<0$, $i=1,...,4$ for stability of the steady state $(N_0,M_0,P_0,Q_0)$ are
\begin{subequations}\label{routh}
\begin{align}
b_1 &> 0,\label{routha}\\
b_4 &> 0,\label{routhb}\\
b_1b_2-b_3 &> 0,\label{routhc}\\
b_1(b_2b_3-b_1b_4)-b_3^2&>0.\label{routhd}
\end{align}
\end{subequations}
Hence, instability of the steady state may arise for some traveling wave speed $v$ if any of the above conditions are violated, i.e.
\begin{equation}\label{instab}
b_1\leq0,\quad b_4 \leq 0,\quad b_1b_2-b_3\leq 0,\quad b_1(b_2b_3-b_1b_4)-b_3^2\leq0.
\end{equation}
In other words, a regime of stability/instability could be created by varying the bifurcation parameter $v$, around the critical values $v_\mp$. Note that the last condition \eqref{routhd} corresponds, at equality, to a Hopf bifurcation with two roots of $(2.3)$ having purely imaginary complex conjugate values. This condition is quartic in $v$ and has the form 
\begin{equation}\label{hopf}
f(v)=v^2(Av^2+B)=0
\end{equation}
where,
\begin{align}\label{AB}
A&=\frac{(D_1+D_2)\big( -\epsilon \alpha G_0+\alpha k F'_0 G_0 +(\epsilon G_0-\beta k F)G'_0\big)}{\alpha \beta k D_1^3D_2^3}\\
B&=\frac{1}{\alpha^2 \beta^2 k^2 D_1^3D_2^3}\Big( \alpha^2 G_0 \big(\epsilon^2 D_2^2 G_0+\beta k (D_1+D_2)^2(\epsilon G_0- \beta k F_0) \big) \notag \\
&+ \big(\alpha k D_2 G_0 F'_0+D_1(\beta k F_0-\epsilon G)G'_0\big)\big(-2 \epsilon \alpha D_2 G_0 \alpha k D_2 G_0 F'_0+D_1(\beta k F_0-\epsilon G)G'_0\big)\Big).\notag \\
\end{align}
The existence of real nonzero roots of \eqref{hopf}, requires the necessary  condition $AB<0$, since on the Hopf curve
\begin{equation}\label{root}
v_\mp=\mp \sqrt{\frac{-B}{A}}.
\end{equation}
The velocity of the traveling wave $v_\mp$ will give the change of stability of the steady state $(N_0,M_0,P_0,Q_0)$ in the following manner:
\begin {itemize}
\item[(a)] If $A>0$, and $B<0$ then the fixed  point is stable in the region $v\in (-\infty,v_{-})\cup (v_{+},\infty)$ and unstable for $v\in (v_{-}, v_{+})$,
\item[(b)] if $A<0$, and $B>0$ then the fixed  point is unstable in the region $v\in (-\infty,v_{-})\cup (v_{+},\infty)$ and stable for $v\in (v_{-}, v_{+})$,
\item[(c)] if $A\geq0$, and $B>0$ then $f(v)>0$, the steady state is stable $\forall v$,
\item[(d)] if $A\leq0$, and $B<0$ then $f(v)<0$, the steady state will be unstable $\forall v$.
\end{itemize}

\section{Hopf bifurcation analysis}

We will perform a Hopf bifurcation analysis \cites{Segel:1, Gierer:2,Keener, Segel:2} to show that as the value of $v$ passes through the critical values $v_\mp$, periodic solutions occur. To determine the behavior of the eigenvalues $\lambda$ as $v$ varies , we use the following lemma.
\begin{lemma}
The characteristic equation \eqref{charac} with  $b_i\in \Re$, $r_i$ roots for $i=1,...,4$ and discriminant given by
\begin{align}\label{discriminant}
\Delta=\prod_{i<j,i\neq j}^4 (r_i-r_j)^2,
\end{align}
has a pair of purely imaginary roots and two real roots only when $\Delta<0$.
\end{lemma}
\begin{proof} At $v=v_\mp$, $b_4=\frac{b_2 b_3}{b_1}-\frac{b_3^2}{b_1^2}$, then the discriminant becomes 
\begin{equation}
\Delta=-\frac{4 b_3 (b_1^3-4 b_1 b_2+4 b_3)(b_1^2b_2^2+b_1^3b_3-4b_1b_2b_3+4 b_3^2)^2}{b_1^6}.
\end{equation}
Since $b_3>0$ then $\Delta<0$ when $b_1^3-4 b_1 b_2+4 b_3>0$. We will see next why this last condition is satisfied. 

Rewriting \eqref{charac} as
\begin{align}\label{gee}
g(\lambda, v)&=\lambda^4+ \frac{(D_1+D_2)v}{D_1D_2} \lambda ^3+ \frac{\alpha(\beta k v^2-\epsilon D_2G_0+kD_2 F'_0G_0)+D_1(\epsilon G_0- \beta k F)G'_0}{\alpha \beta k D_1D_2} \lambda ^2 \\ \notag
&+ \frac{v\big( -\epsilon \alpha G_0+\alpha k F'_0 G_0 +(\epsilon G_0-\beta k F)G'_0\big)}{\alpha \beta k D_1D_2} \lambda + \frac{(\beta k F_0-\epsilon G_0)\big( \alpha \beta k +(\epsilon-k F_0)G'_0\big)G_0}{\alpha \beta^2 k^2 D_1D_2},
\end{align}
and evaluating this on the Hopf curve, i.e., at $v=v_\mp$ yields to
\begin{equation}\label{charachopf}
g(\lambda, v_\mp)=(\lambda^2+ \omega^2)(\lambda^2+s \lambda+p),
\end{equation}
where,
\begin{align}\label{roots}
r_{1,2}&=\mp i\omega=\mp i\sqrt{\frac{\bar{b}_3}{\bar{b}_1}}\\ \notag
r_{3,4}&=\frac 12 \Big(\bar{b}_1 \pm \sqrt{\frac{\bar{b}_1^3-4 \bar{b}_1 \bar{b}_2+4 \bar{b}_3}{\bar{b}_1}} \Big)
\end{align}
and $\bar{b}_i=b_i(v_\mp)$. Since we require that $b_1>0$ by \eqref{routha}, and $r_{3,4}\in\Re$,  then $b_1^3-4 b_1 b_2+4 b_3>0$. This leads to a Hopf bifurcation setting as evidenced by the pair of imaginary eigenvalues $r_{1,2}$ that oscillate with frequency 
\begin {equation} \label{omega}
\omega=\sqrt{\frac{G'_0(\epsilon G_0- \beta k F_0)+\alpha G_0(k F'_0-\epsilon)}{\alpha \beta k(D_1+D_2)}}.
\end{equation}
\end{proof}
In order to introduce the relevant notation, we state the Hopf bifurcation theorem.
\begin{theorem}
Let
\begin{equation} \label{hopftheo}
\frac{d\vec{x}}{dt}=\vec{F}(\vec{x},\mu)
\end{equation}
be an autonomous system of differential equations for each value of the parameter $\mu \in (-\mu_0,\mu_0)$, where $\mu_0$ is a positive number and the vector function $\vec{F} \in C^2(D\times (-\mu_0,\mu_0))$, where $D$ is a domain in $\Re^n$. Suppose that the system \eqref{hopftheo} has a critical point $x_0(\mu)$, that is,
\begin{equation}
\vec{F}(\vec{x}_0(\mu),\mu)=0.
\end{equation}
Let $\vec{J}(\mu)$ be the Jacobian matrix of system \eqref{hopftheo} at $x_0(\mu)$. Suppose that $det(\vec{J}(\mu)-\lambda I)=0$ has a complex conjugate pair of solutions $\lambda(\mu),\lambda^*(\mu)$ such that for $\mu>0$, $Re\lambda(\mu)>0$; $\mu=0$, $Re\lambda(\mu)=0$; $\mu<0$, $Re\lambda(\mu)<0$; and $\frac{d Re \lambda(\mu)}{d \mu}|_{\mu=0}>0$.  Assuming that all other $\lambda$'s are distinct, \eqref{hopftheo} has a periodic solution in some neighborhood of $\mu=0$ and $\vec{x}$ in some neighborhood of $\vec{x}_0(\mu)$.
\end{theorem}
\begin{proof}
In order to apply this theorem to \eqref{gee}, we define the bifurcation parameter
\begin{equation}
	\mu=\frac 1 v -\frac{1}{v_0},
\end{equation} then
\begin{equation}
	v(\mu)=\frac{v_0}{1+v_0 \mu}
\end{equation}
with $\mu=0$ at $v=v_0$. For $\mu>0$, $v<v_0$ and $Re\lambda(\mu)>0$. For $\mu=0$, $v=v_0$ and  $Re\lambda(\mu)=0$, and for  $\mu<0$,  $v>v_0$ and $Re\lambda(\mu)<0$. Thus, the first set of conditions in the theorem are valid, and it remains only to show that
\begin{equation}
\frac{d Re \lambda(v(\mu))}{d \mu}|_{\mu=0}=\frac{d Re \lambda(v)}{d v}\frac{dv}{d \mu}|_{\mu=0}>0.
\end{equation}
For $v_0=v_{\mp}$, then $g(\lambda(v_0),v_0)=g(\mp i \omega,v_0)=0$ by \eqref{charachopf}. Hence, implicitly differentiating $g(\lambda(v),v)=0$, \eqref{gee}, yields:
\begin{equation}
\frac{d \lambda}{d v}=-\frac{\frac{\delta g}{\delta v}}{\frac{\delta g}{\delta \lambda}}=\frac{\omega(\bar{b}_3^\prime-\omega ^2 \bar{b}_1^\prime+i \omega \bar{b}_2^\prime)}{2\omega(2\omega^2-\bar{b}_2)+i(\bar{b}_3-3 \omega ^2\bar{b}_1)},
\end{equation}
where $\bar{b}_i^\prime=\frac{d b_i}{dv}|_{v=v_0}$,
therefore
\begin{equation}
\frac{d Re \lambda(v_0)}{d v}=\frac{\omega^2 \Psi}{4\omega^2(2\omega^2-\bar{b}_2)^2+(\bar{b}_3-3 \omega ^2\bar{b}_1)^2},
\end{equation}
where
\begin{equation}
\Psi=2(2 \omega ^2-\bar{b}_2)(\bar{b}_3^\prime-\omega ^2 \bar{b}_1^\prime)+\bar{b}_2^\prime(\bar{b}_3-3 \omega^2\bar{b}_1).
\end{equation}
Evaluating the required derivatives of $b_i$'s at $v_0$ and using \eqref{omega} yields
\begin{equation}
\Psi=\frac{4 v_0^2}{D_1^2 D_2^2}\frac{G'_0(-\epsilon G_0+ \beta k F_0)+\alpha G_0(-k F'_0+\epsilon)}{\alpha \beta k}.
\end{equation}
Thus, assuming that $D_1+D_2>0$, and using  \eqref{omega}, then $\Psi<0$, and hence $\frac{d Re \lambda(v_0)}{d v}<0$. Since $\frac{dv}{d \mu}=-\frac{v_0^2}{(1+v_0 \mu)^2}<0$, then $\frac{d Re \lambda(v(\mu))}{d \mu}|_{\mu=0}>0$. \\

All the conditions of the Hopf bifurcation theorem are satisfied. Thus, Hopf bifurcations occur and periodic solutions will exist in the neighborhood of $v_0$.
\end{proof}

\section{Contracting/Dilatory behavior}

The stability of the bifurcating closed orbits may be investigated for each the specific choice of nonlinearity $F(N)$ and $G(P)$, by reducing the system to the center manifold (since one has  two purely imaginary eigenvalues) as done in [28]. This will not be considered in here. Instead, we shall consider numerical solutions of \eqref{4D} in the following section, which will allow both the verification of the preceding analysis and also yield more quantitative results.

We will concentrate on the five specific choices of $F(N)$ and $G(P)$  referred to us in this paper as systems A--E. For all models, the local rate of change of volume of the $(N,M,P,Q)$ phase--space in the vicinity of the fixed points $(N_0,M_0,P_0,Q_0)$, which gives the local logarithmic rate of change of $(N,M,P,Q)$ phase--space volume $V$ is given by the trace of the Jacobian matrix of \eqref{jaco} at the fixed points, where $Tr(J)=\frac1V\frac{\mathrm{d}V}{\mathrm{d}t}=-\frac{(D_1+D_2)v}{D_1D_2}\equiv -b_1$. A necessary condition for the stability of the steady state is that $b_1>0$ by \eqref{routha}, therefore models that start from stable/unstable fixed points (depending upon one or more of \eqref{routhb}--\eqref{routhd} is violated) will be locally dissipative, i.e., (phase--space volumes contract), so we may anticipate that the orbits may go to an attractor at infinity if the dissipation is weak, or dilatory (volumes expand) if \eqref{routha} is violated. If the fixed point is stable, the predator population is ultimately decimated, i.e., $k$ and the rate of conversion $\beta$ of prey into predator are not large enough to sustain the predator population. If the fixed point is unstable, for a parameter regime where the  system is strongly dissipative, one might anticipate possible bounded chaotic dynamics evolving on a strange attractor. This will be tested numerically in the next section.

\subsection{System A}
Using $F(N)=\epsilon$, $G(P)=\gamma$, and $\tilde{\epsilon}=0$, \eqref{4D} becomes
\begin{align}\label{4DA}
\dot{N}&=M\\
\dot{M}&=\frac{1}{D_1}\big(-vM-\epsilon N+\alpha NP \big)\notag \\
\dot{P}&=Q \notag \\
\dot{Q}&=\frac{1}{D_2}\big(-vQ+\gamma P-\beta NP\big), \notag 
\end{align}
with equilibrium points $(N_0,M_0,P_0,Q_0)=\big(\frac{\gamma}{\beta},0,\frac{\epsilon}{\alpha},0\big)$. The characteristic equation \eqref{charac} has coefficients 
\begin{align}\label{bA}
b_1&=\frac{(D_1+D_2)v}{D_1D_2}\\
b_2&=\frac{v^2}{D_1D_2}\notag \\
b_3&=0\notag \\
b_4&=\frac{\epsilon \gamma}{D_1D_2}. \notag 
\end{align}
Therefore, the Hopf curve \eqref{hopf} is
\begin{equation}\label{HopfA}
f(v)=-\frac{(D_1+D_2)^2\epsilon \gamma}{D_1^3D_2^3}v^2 \equiv 0,
\end{equation}
hence, the bifurcation parameter is 
\begin{equation}\label{vA}
v_\mp=0.
\end{equation}
The characteristic polynomial \eqref{charac} evaluated at the fixed point and on the Hopf curve \eqref{vA} has the form
\begin{equation}\label{charachopfA}
g(\lambda, v_\mp)=\lambda^4 +\bar{b}_4,
\end{equation}
and $\bar{b}_4=b_4(v_\mp)$.

\subsection{System B}
Using $F(N)=\epsilon$, $G(P)=\gamma$, and $\tilde{\epsilon}=\epsilon$, \eqref{4D} becomes
\begin{align}\label{4DB}
\dot{N}&=M\\
\dot{M}&=\frac{1}{D_1}\big(-vM-\epsilon N+\alpha NP+\frac{\epsilon N^2}{k} \big)\notag \\
\dot{P}&=Q \notag \\
\dot{Q}&=\frac{1}{D_2}\big(-vQ+\gamma P-\beta NP\big), \notag 
\end{align}
with equilibrium points $(N_0,M_0,P_0,Q_0)=\big(\frac{\gamma}{\beta},0,\frac{\epsilon}{\alpha}(1-\frac{\gamma}{\beta k}),0\big)$. The characteristic equation \eqref{charac} has coefficients 
\begin{align}\label{bB}
b_1&=\frac{(D_1+D_2)v}{D_1D_2}\\
b_2&=\frac{\beta k v^2-\epsilon \gamma D_2}{\beta k D_1D_2}\notag \\
b_3&=-\frac{\epsilon \gamma v}{\beta k D_1D_2}\notag \\
b_4&=\frac{\epsilon \gamma(\beta k- \gamma)}{\beta k D_1D_2}. \notag 
\end{align}
Therefore, on the Hopf curve, the bifurcation parameter is 
\begin{equation}\label{vB}
v_\mp=\mp \sqrt{\frac{\epsilon \gamma D_2^2+\beta k(D_1+D_2)^2(\gamma-\beta k)}{\beta k (D_1+D_2)}},
\end{equation} 
where 
\begin{equation}\label{HopfB}
f(v)=\frac{\epsilon \gamma v^2}{\beta ^2 k ^2 D_1^3D_2^3}\big(-\beta k (D_1+D_2)v^2+\epsilon \gamma D_2^2+\gamma \beta k (D_1+D_2)^2-\beta^2 k^2 (D_1+D_2)^2\big).
\end{equation}
The characteristic polynomial \eqref{charac} evaluated at the fixed point and on the Hopf curve \eqref{vB} has the form
\begin{equation}\label{charachopfB}
g(\lambda, v_\mp)=\lambda^4 +\bar{b}_1 \lambda ^3+\bar{b}_2 \lambda ^2+\bar{b}_3 \lambda+\bar{b}_4,
\end{equation}
where $\bar{b}_i=b_i(v_\mp)$.

\subsection{System C}
Using $F(N)=k_0$, $G(P)=d+c P$, and $\tilde{\epsilon}=\epsilon$, \eqref{4D} becomes
\begin{align}\label{4DC}
\dot{N}&=M\\
\dot{M}&=\frac{1}{D_1}\big(-vM-k_0 N+\alpha NP+\frac{\epsilon N^2}{k} \big)\notag \\
\dot{P}&=Q \notag \\
\dot{Q}&=\frac{1}{D_2}\big(-vQ+d P-\beta NP+ c P^2\big), \notag 
\end{align}
with equilibrium points $(N_0,M_0,P_0,Q_0)=(\frac{k(\alpha d+ c k_0)}{\epsilon c+ \alpha \beta k},0,\frac{\beta k k_0-\epsilon d}{\epsilon c + \alpha \beta k},0)$. The characteristic equation \eqref{charac} has coefficients 
\begin{align}\label{bC}
b_1&=\frac{(D_1+D_2)v}{D_1D_2}\\
b_2&=\frac{\alpha( \beta k v^2 -\epsilon d D_2)+c\big(\epsilon d D_1 -(\epsilon D_2+\beta k D_1)k_0+\epsilon v^2\big)}{(\epsilon c +\alpha  \beta k )D_1D_2}\notag \\
b_3&=\frac{v\big(\epsilon d (c-\alpha )-c
(\epsilon +\beta k)\big)}{(\epsilon c +\alpha  \beta k )D_1D_2}\notag \\
b_4&=\frac{(\alpha d + c k_0)(\beta k k_0-\epsilon d)}{(\epsilon c +\alpha  \beta k )D_1D_2}. \notag 
\end{align}
Therefore, on the Hopf curve, the bifurcation parameter is 
\begin{equation}\label{vC}
v_\mp=\mp \sqrt{\frac{c^2(\epsilon d D_1+\epsilon k_0 D_2-\beta k k_0 D_1)^2+\alpha^2 d \big(d\epsilon^2D_2^2+\beta k(\epsilon d-\beta k k_0)(D_1+D_2)^2\big)+\Theta}{(\epsilon c+\alpha \beta k)(D_1+D_2)\big(\epsilon d(\alpha-c)+c(\epsilon+\beta k)k_0\big)}}
\end{equation} 
where 
\begin{align}\label{HopfC}
f(v)&=\frac{v^2}{(\epsilon c + \alpha \beta k)^2D_1^3D_2^3}\Big((\epsilon c+\alpha \beta k)(D_1+D_2)\big((\alpha D+c k_0)(\epsilon d- \beta k k_0)(D_1+D_2)\\ \notag
& +\frac{\big(\epsilon d(\alpha -c)+c(\epsilon+\beta k)k_0)\big)\big(c(\beta k k_0-\epsilon d)D_1+\epsilon(\alpha d+c k_0)D_2-(\epsilon c+\alpha \beta k)v^2\big)}{\epsilon c + \alpha \beta k}\big)\\ \notag
& -D_1D_2\big(\epsilon d (\alpha-c)+ck_0(\epsilon+\beta k)\big)^2\Big),
\end{align}
and $\Theta=\alpha c \Big(2 \epsilon^2 d^2 D_1 D_2+\epsilon d k_0\big(2\epsilon D_2^2+\beta k (D_1^2+D_2^2)\big)-\beta^2 k^2 k_0^2(D_1+D_2)^2\Big).$

In this case, the characteristic polynomial \eqref{charac} evaluated at the fixed point and on the Hopf curve \eqref{vC} has the same form as \eqref{charachopfB}.

\subsection{System D}
Using $F(N)=k_0(1+\frac{N}{k})$, $G(P)=d+c P$, and $\tilde{\epsilon}=\epsilon$ ,\eqref{4D} becomes
\begin{align}\label{4DD}
\dot{N}&=M\\
\dot{M}&=\frac{1}{D_1}\Big(-vM-k_0 N+\alpha NP+\frac{(\epsilon-k_0) N^2}{k} \Big)\notag \\
\dot{P}&=Q \notag \\
\dot{Q}&=\frac{1}{D_2}\Big(-vQ+d P-\beta NP+ c P^2\Big), \notag 
\end{align}
with equilibrium points $(N_0,M_0,P_0,Q_0)=(\frac{k(\alpha d+ c k_0)}{(\epsilon-k_0) c+ \alpha \beta k},0,\frac{\beta k k_0-(\epsilon-k_0) d}{(\epsilon-k_0) c + \alpha \beta k},0)$. 
This system is not quantitatively different form System C, therefore all the equations \eqref{bC},\eqref{vC}, and \eqref{HopfC} will stay the same as long as we replace $\epsilon\rightarrow\epsilon-k_0$. 
\subsection{System E}
Using, $F(N)=\frac{1+\delta N}{1+N^2}$, $G(P)=\gamma(1+k P^2)$, and $\tilde{\epsilon}=0$, \eqref{4D} becomes
\begin{align}\label{4DE}
\dot{N}&=M \\
\dot{M}&=\frac{1}{D_1}\Big(-vM-\frac{N(1+\delta N)}{1+N^2}\alpha NP\Big)\notag \\
\dot{P}&=Q \notag \\
\dot{Q}&=\frac{1}{D_2}\Big(-vQ+\gamma P(1+k P^2)-\beta NP\Big). \notag 
\end{align}
The equilibrium points are found numerically by solving the following system which involves a quintic algebraic equation in $N_0$.
\begin{align}
N_0&=\frac{\gamma}{\beta}(1+kP_0^2) \\
M_0&=0\\ \notag
P_0&=\frac{1+\delta N_0}{\alpha(1+N_0^2)}\\ \notag
Q_0&=0. \notag
\end{align}
The characteristic equation \eqref{charac} has coefficients 
\begin{align}\label{bE}
b_1&=\frac{(D_1+D_2)v}{D_1D_2}\\ \notag
b_2&=\frac{1}{D_1D_2}\Big[v^2+\beta N_0 D_1-\gamma D_1(1+3 k P_0^2)\\ \notag
&+\frac{D_2}{(1+N_0^2)^2}\big(1-\alpha P_0-N_0(N_0+\alpha N_0 P_0(2+N_0^2)-2\delta)\big)\Big]\\ \notag
b_3&=\frac{v}{D_1D_2(1+N_0^2)^2}\Big[\beta N_0^5+\big(\gamma+P_0(\alpha+3k \gamma P_0)\big)N_0^4+2\beta N_0^3+1-P_0(\alpha+3 k \gamma P_0)\\ \notag
&-\gamma +N_0(\beta +2 \delta)-N_0^2\big(1+2 \gamma+2P_0(\alpha+3 k \gamma P_0)\big)\Big] \\ \notag
b_4&=\frac{1}{D_1D_2}\Big[\alpha \beta N_0 P_0 -\frac{\big(\beta N_0-\gamma(1+3kP_0^2)\big)\big(\alpha P_0 -1+N_0(N_0-2\delta+\alpha N_0 P_0(2+N_0^2))\big)}{(1+N_0^2)^2}\Big].\notag 
\end{align}
Therefore, on the Hopf curve, the bifurcation parameter $v_\mp$ can only be found numerically by  \eqref{root}, where
\begin{align}
A&=\frac{D_1+D_2}{D_1^3 D_2^3(1+N_0^2)^2}\Big[\beta N_0^5+\big(\gamma+P_0(\alpha+3k \gamma P_0)\big)N_0^4+2\beta N_0^3+1-P_0(\alpha+3 k \gamma P_0)\\ \notag
&-\gamma+N_0(\beta +2 \delta)-N_0^2\big(1+2 \gamma+2P_0(\alpha+3 k \gamma P_0)\big)\Big] \\ \notag
B&=\frac{1}{D_1^3 D_2^3(1+N_0^2)^4}\Big[D_1^2(1+N_0^2)^4T_1+2D_1D_2(1+N_0^2)^2T_2+D_2^3 T_3\Big],\\ \notag
\end{align}
and
\begin{align}
T_1&=\beta ^2 N_0^2+(\gamma +3 \gamma k P_0^2)^2-\beta N_0\big(2\gamma +P_0(\alpha+6 \gamma kP_0)\big)  \\ \notag
T_2&=\alpha \gamma P_0(1+3k P_0^2)N_0^4-\beta N_0^3+\big(2\beta \delta +(1+3k P_0^2)(1+2 \alpha P_0)\big)N_0^2 \\ \notag
&+\big(\beta-2\gamma \delta(1+3k P_0^2)\big)N_0-\gamma(1+3k P_0^2)(1-\alpha P_0)\\ \notag
T_3&=-\alpha \beta P_0 N_0^9+\alpha^2 P_0^2 N_0^8-4 \alpha \beta P_0 N_0^7+2\alpha P_0(1+2\alpha P_0) N_0^6-2\alpha (3 \beta+2 \delta)P_0N_0^5\\ \notag
&+\big(1+2\alpha P_0(1+3\alpha P_0)\big)N_0^4-4\big(\delta + \alpha(\beta +2 \delta)P_0\big)N_0^3+2\big(2\delta^2-1-\alpha(1-2\alpha P_0)P_0\big)N_0^2 \\ \notag
&+\big(4 \delta -\alpha P_0(\beta +4 \delta)\big)N_0+(1-\alpha P_0)^2.
\end{align}

\section{Numerical results}

For the numerical results we will concentrate on our five systems, choosing for each system specific parameters that will show the dissipative or dilatory behavior.
\subsection{System A}
We choose parameters such that $\bar{b}_4=\frac{\epsilon \gamma}{D_1D_2}<0$. Since \eqref{routhb} is violated it means that we start from an unstable fixed point in a constant volume space. In this case, $f(v)>0, \forall v$, therefore the steady state remains stable since both populations will annihilate.  For the system parameters of $\alpha=1.5$, $\epsilon=1$, $\beta=-2.75$, $\gamma=-1.5$, $D_1=1.25$, $D_2=2.1$, and the bifurcation parameter $v=0.1$, then $\bar{b}_4=\frac{\epsilon\gamma}{D_1D_2}=-0.571429<0$, and the populations start to oscillate from the stationary point $(N_0,M_0,P_0,Q_0)=(0.55,0,0.67,0)$ with frequency  $\omega=\sqrt{-\bar{b}_4}=0.755$. Because the space is contracting, since $\bar{b}_3=0.128$, eventually both predators/prey populations will assimilate each other and reach the stable equilibrium null populations. This attenuating  behavior is presented in Fig.\ref{FigA1}. Note the stable periodic oscillations on the stable limit cycle created by a supercritical Hopf bifurcation at $v=v_{\mp}=0$. If we were to increase $v$ then the population would have terminated much faster.
\begin{figure}[!ht]
					\begin{center} 
\includegraphics[width=0.65\textwidth]{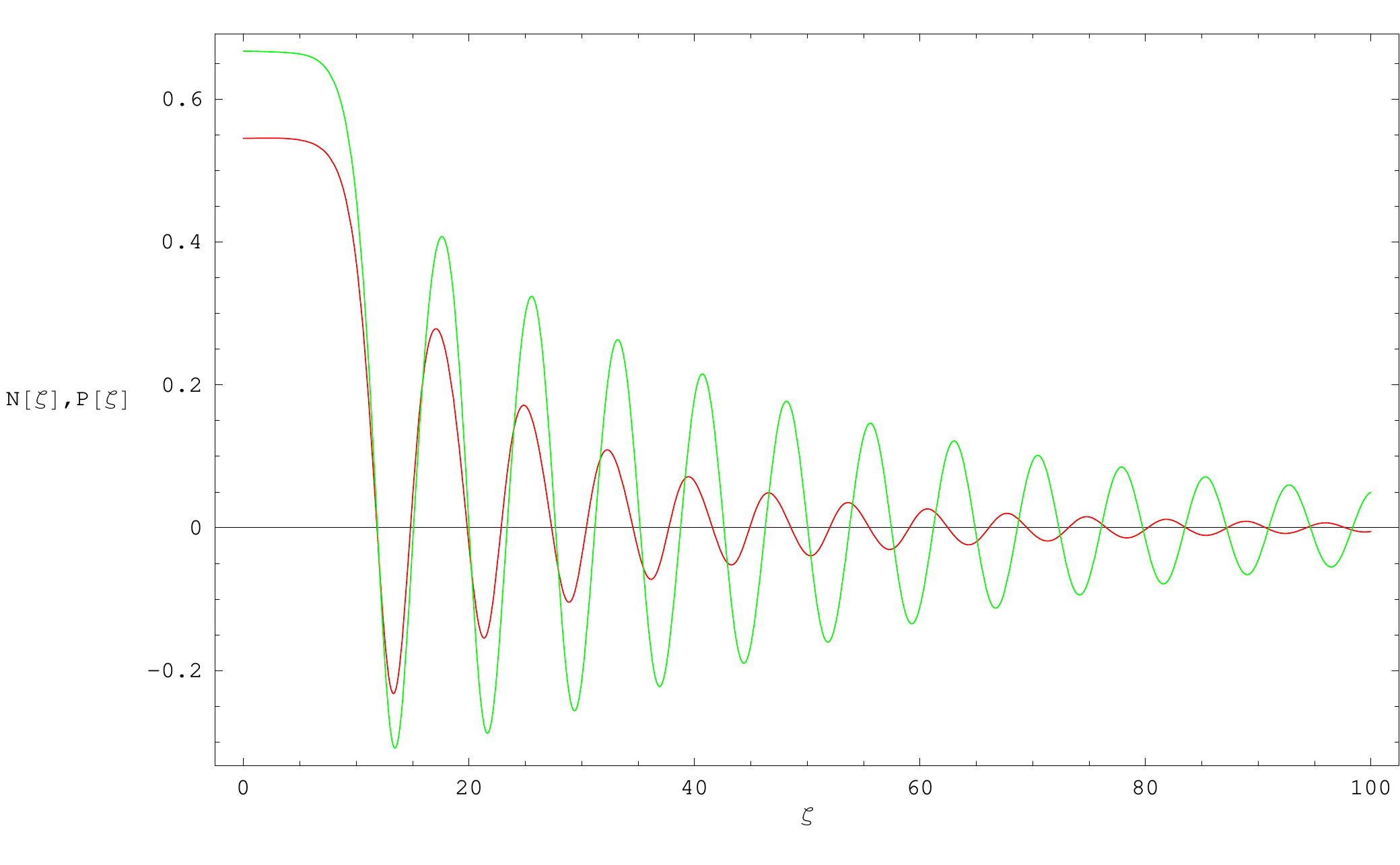}
					\end{center}
				\caption{\small{Periodic time series for populations of System A, $v=0.1$}}\label{FigA1}
\end{figure}
\subsection{System B}

For the system parameters given by the set $\alpha=-1.2$, $\epsilon=-3$, $\beta=-2$, $\gamma=-2$, $D_1=1$, $D_2=2$, $k=2$, the Hopf velocity is $v_\mp=\mp2$. Therefore, the populations start to oscillate from any stationary point with frequency $\omega=\sqrt{\frac{\bar{b}_3}{\bar{b}_1}}=\frac{\sqrt 2}{2}$.  To find the regimes when the fixed point changes stability, we find the coefficients of the Hopf condition \eqref{hopf}, which are $A=\frac {9}{16}$, and $B=-\frac94$, and we analyze $f(v)$. Hence, the fixed point $(N_0,M_0,P_0,Q_0)=(1,0,1.25,0)$ is stable in the region $v\in (-\infty,-2)\cup (2,\infty)$ and unstable for $v\in (-2, 2)$. 

Since the volume of the system is expansive on $v_{-}=-2$, and contractive on $v_{+}=2$, as we vary $v$ around $v_{\mp}$ we will expect different behavior on both sides of the bifurcation parameter. In a contracting space, $v_{+}=2$, then $\bar{b}_1=3,\bar{b}_2=3.5,\bar{b}_3=1.5,\bar{b}_4=1.5 $, and hence the population will oscillate from any fixed point toward the equilibrium $(N_0,M_0,P_0,Q_0)=(1,0,1.25,0)$.  

When $v=2.2$, the fixed point remains stable, hence the populations dissipate as in case A, but instead of reaching the null populations they will converge towards nonzero equilibrium values. This behavior is shown in Fig. \ref{FigB1}. 

If $v=1.9$, the fixed point becomes unstable, and, after an initial transient, both populations settle onto the stable limit cycle created by the supercritical Hopf bifurcation. The corresponding spatially periodic wavetrain in spatial variable $\zeta$ is shown in Fig. \ref{FigB2}.
\begin{figure}[!ht]
					\begin{center} 
\includegraphics[width=0.7\textwidth]{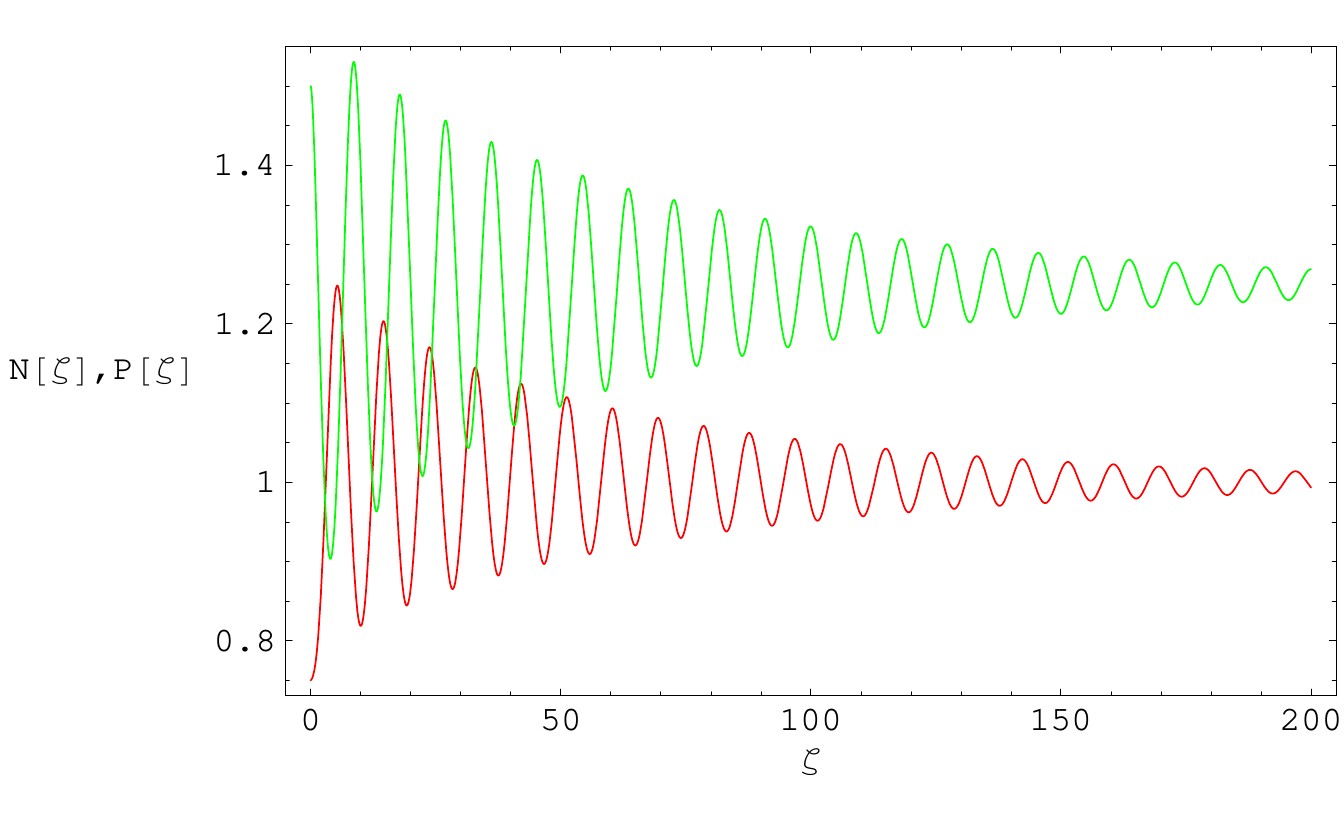}
					\end{center}
				\caption{\small{Periodic evolution for populations of System B, $v=2.2$}}\label{FigB1}
\end{figure}
\begin{figure}[!ht]
					\begin{center} 
\includegraphics[width=0.7\textwidth]{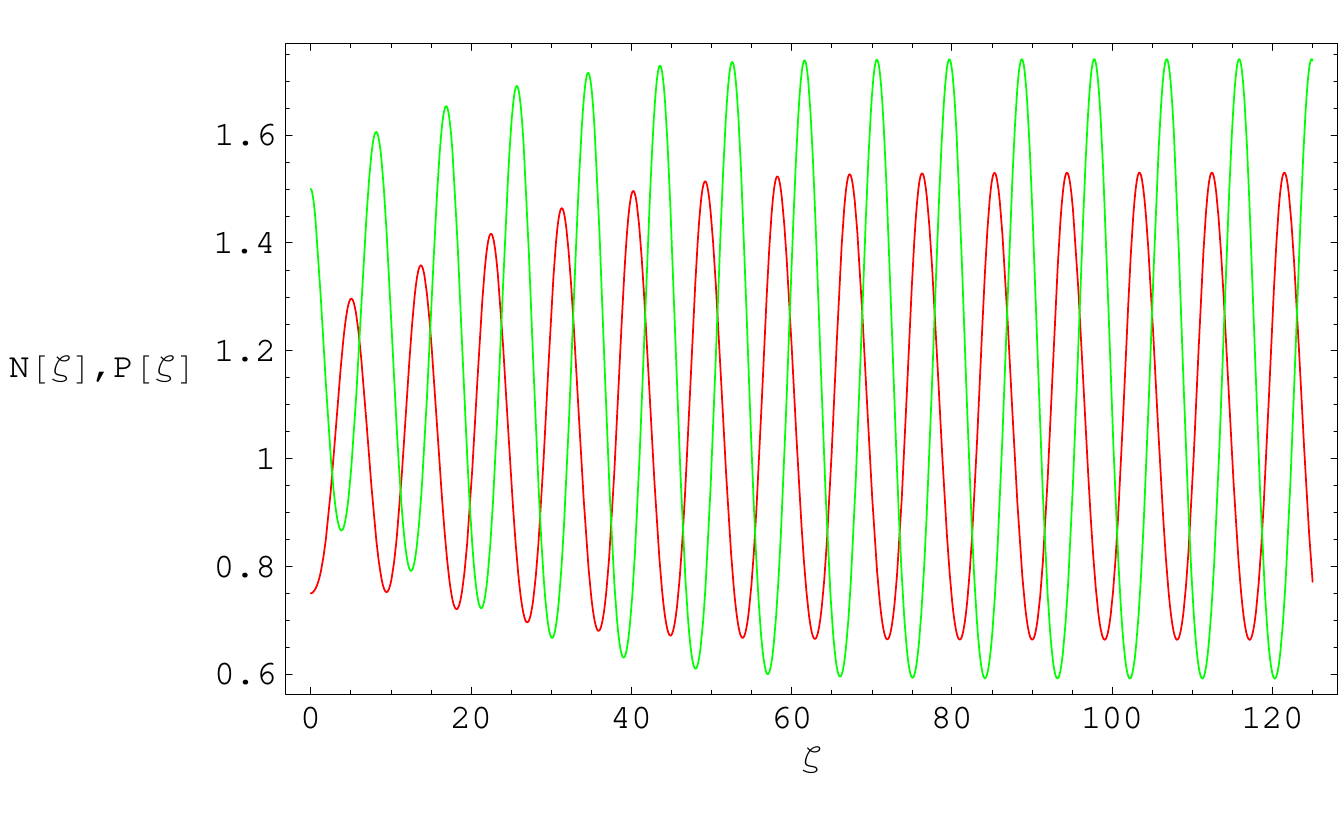}
					\end{center}
				\caption{\small{Periodic evolution for populations of System B, $v=1.9$}}\label{FigB2}
\end{figure}

By contrast, on the left side of the bifurcation parameter, the system is expansive or dilatory at $v_{-}=-2$ and undergoes a subcritical Hopf bifurcation which occurs at $v=v_{-}$. This corresponds to an unstable periodic orbit coexisting with an unstable fixed point $(N_0,M_0,P_0,Q_0)=(1,0,1.25,0)$, since \eqref{routhc} is violated. For this case, $\bar{b}_1=-3$, $\bar{b}_2=3.5$, $\bar{b}_3=-1.5$, $\bar{b}_4=1.5$. Because the system is expanding then the only possibility is to have an attractor at infinity. Hence, the populations blow at a finite value of $\zeta$.

\subsection{System C/System D}
Since these two systems are similar as explained in previous section, for numerical simulations we will describe the behavior of only System D. Choosing the system parameters given by the set $\alpha=1.25$, $\epsilon=1$, $\beta=2$, $c=0.5$, $d=2$, $k=2$, $k_0=2$, $D_1=1$ and $D_2=-2$, the equilibrium point is $(N_0,M_0,P_0,Q_0)=(1.55,0,2.22,0)$, while the bifurcation parameter on the Hopf curve is $v_{\mp}=\mp5.03$. Here, $\bar{b}_1=-2.51$, $\bar{b}_2=-11.22$, $\bar{b}_3=-0.83$, 
$\bar{b}_4=-3.88 $. For these values, \eqref{routhc} is not violated but  \eqref{routha} and  \eqref{routhb} are, hence the fixed point is unstable. To find the regimes when the fixed point changes stability we find the coefficients of the Hopf condition \eqref{hopf}, $A=-0.04$, and $B=1.05$, and we analyze $f(v)$. Hence, the fixed point will remain unstable in the region $v\in (-\infty,-5.03)\cup (5.03,\infty)$ and stable for $v\in (-5.03, 5.03)$.
Within the stable region if $v=-0.2$, the volume is weakly expanding hence we anticipate that the obits may go to an attractor at infinity. From Fig. \ref{FigD1} we can see the aperiodic behavior of the populations. The orbits fly off to an attractor at infinity as shown in Fig. \ref{FigD2} by both $N(\zeta)$, and $P(\zeta)$ blowing up around $\zeta=96$. 

\begin{figure}[!ht]
					\begin{center} 
\includegraphics[width=0.6\textwidth]{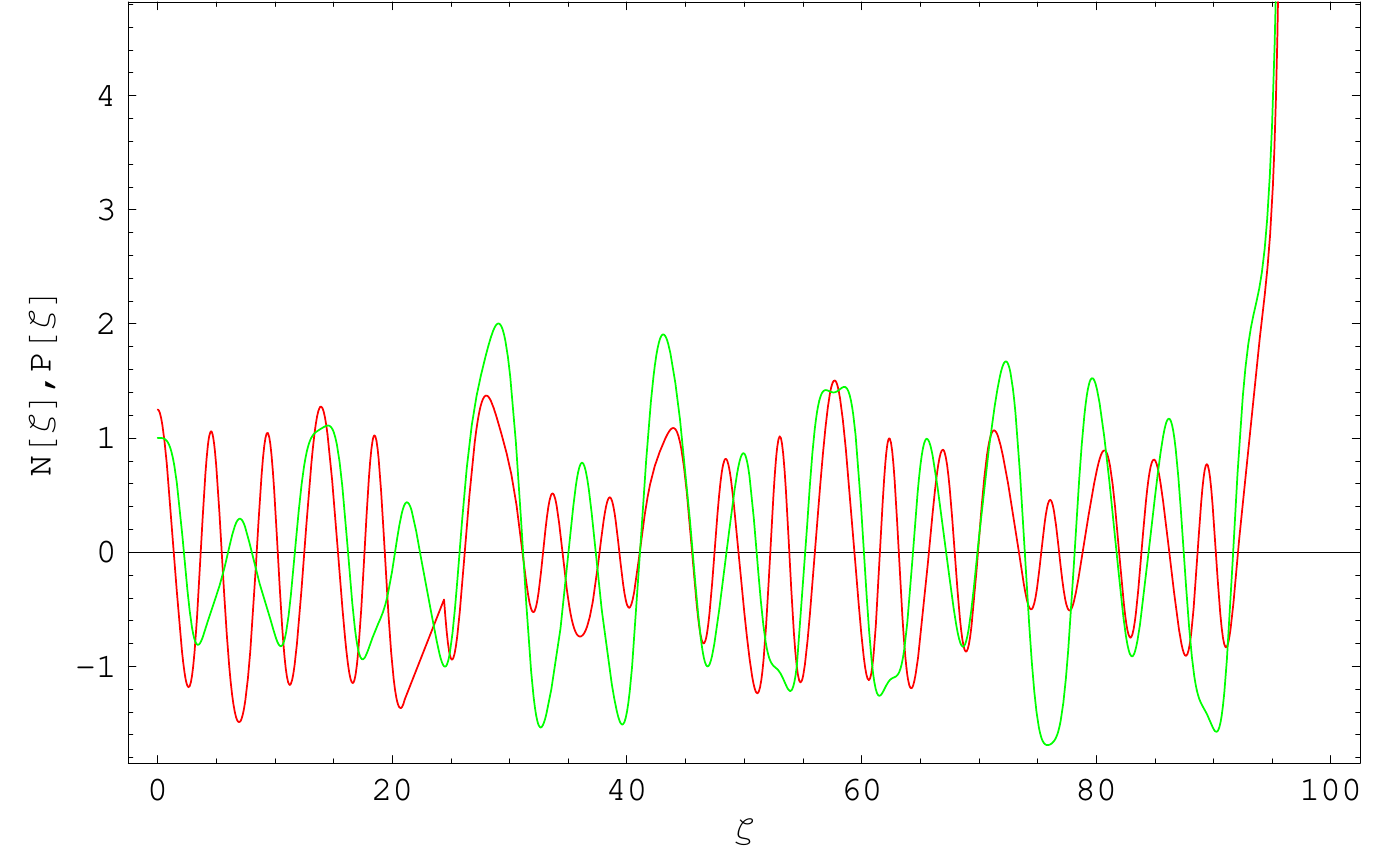}
					\end{center}
				\caption{\small{Aperiodic evolution for populations of System D, $v=-0.2$}}\label{FigD2}
\end{figure}
\begin{figure}[!ht]
					\begin{center} 
\includegraphics[width=0.5\textwidth]{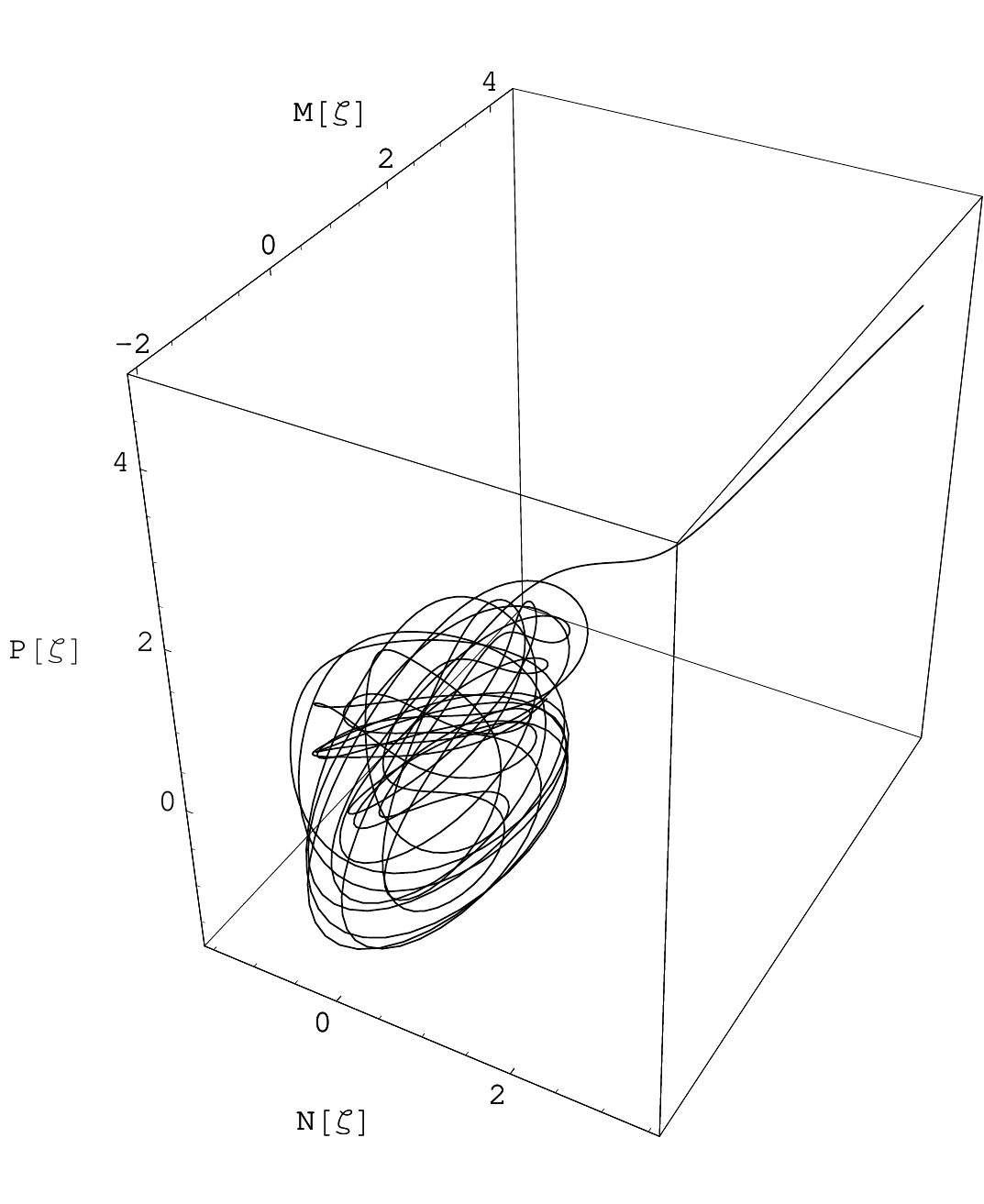}
					\end{center}
				\caption{\small{Attractor at infinity for populations of System D, $v=-0.2$}}\label{FigD1}
\end{figure}

If $v=1.2$ the volume is dissipative but the fixed point is unstable, hence the populations experience qusiperiodic behavior or bounded chaotic behavior. We will present this case next.

\subsection{System E}
Since in this case the fixed points can not be found analytically, due to a quintic algebraic equation, we will solve this case completely numerically. For the parameters set $\alpha=1.7$, $\beta=-2.1$, $\gamma=-2$, $\delta=0.6$, $k=-2$, $D_1=-1$ and $D_2=2$, the equilibrium point is $(N_0,M_0,P_0,Q_0)=(0.19,0,0.63,0)$, while the bifurcation parameter on the Hopf curve is $v_{\mp}=\mp1.8$. Here, $\bar{b}_1=-0.89$, $\bar{b}_2=-3.252$, $\bar{b}_3=2.84$, 
$\bar{b}_4=0.27 $. For these values, \eqref{routhc} is not violated but \eqref{routha} is, therefore the fixed point is unstable. To find the regimes when the fixed point changes stability we find the coefficients of the Hopf condition \eqref{hopf}, $A=0.39$, and $B=-1.027$, and we again analyze $f(v)$. Hence, the fixed point will become stable in the region $v\in (-\infty,-1.8)\cup (1.8,\infty)$ and stable for $v\in (-1.8, 1.8)$. For $v=-1.1$ the volume is dissipative and as explained above the populations will behave chaotically. Fig \ref{FigE1} shows the numerical solutions for  $N(\zeta)$ and  $P(\zeta)$ vs. the spatial variable $\zeta$. Notice the strange aperiodic dynamics. Note that unlike Fig. \ref{FigD2} the solution remains unbounded for all time. The 3D phase space plot in the $(N,M,P)$ space is shown in Fig. \ref{FigE2}. Notice that the solutions retrace the same region of phase space repeatedly, suggesting bounded chaotic dynamics on an attractor. 
\begin{figure}[!ht]
					\begin{center} 
\includegraphics[width=0.5\textwidth]{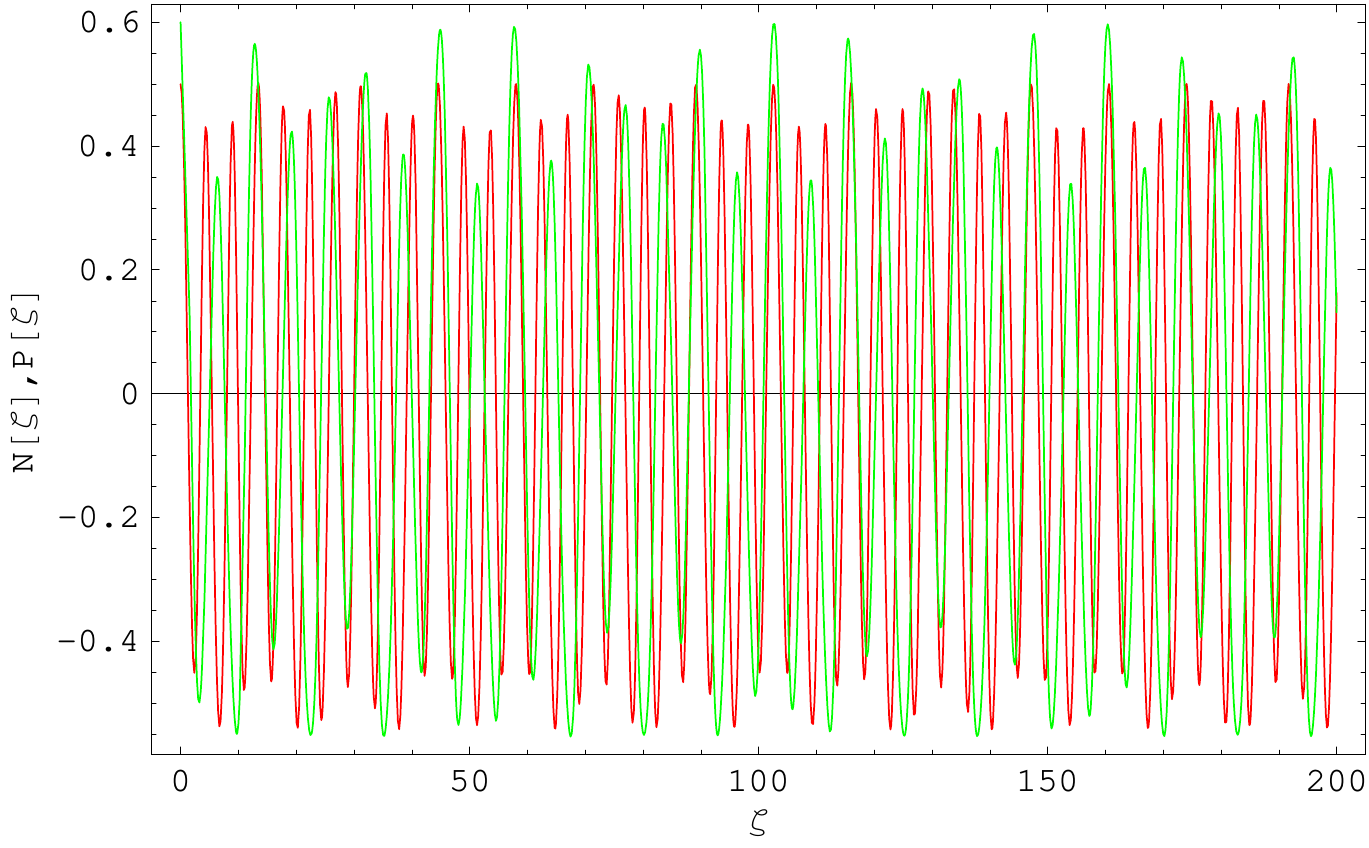}
					\end{center}
				\caption{\small{Aperiodic evolution for populations of System E, $v=-1.1$}}\label{FigE1}
\end{figure}
\begin{figure}[!ht]
					\begin{center} 
\includegraphics[width=0.5\textwidth]{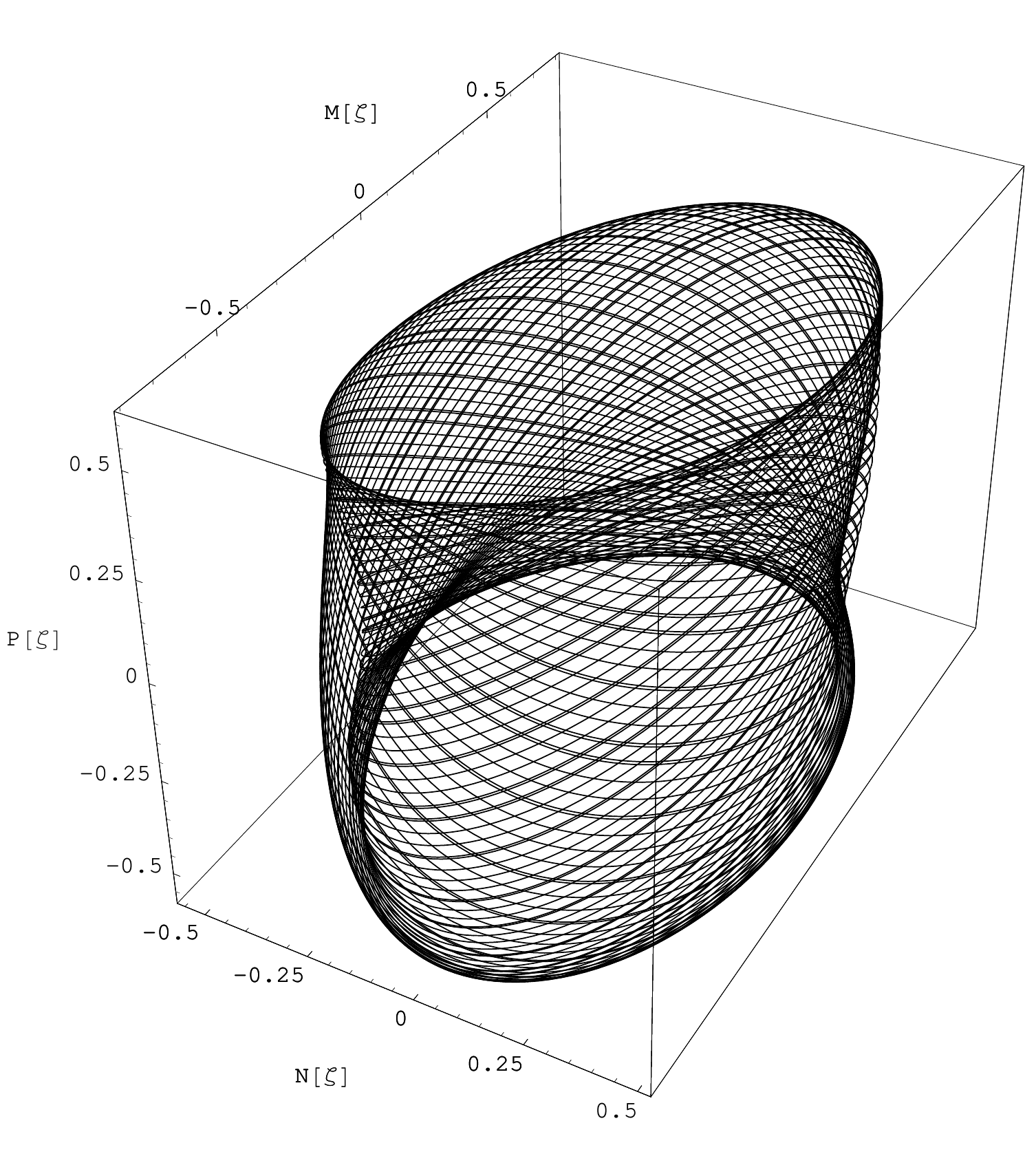}
					\end{center}
				\caption{\small{Attractor for populations of System E, $v=-1.1$}}\label{FigE2}
\end{figure}

In order to confirm this and further characterize the suspected chaotic solutions, we employ the standard numerical diagnostics \cites{Nayfeh,Grassberger} i.e., the power spectral density, the autocorrelation function, and the fractal dimensions. The power spectral density and the autocorrelation function of $N(\zeta)$ are computed using codes from "Numerical Recipes in C'' \cite{Press}, and the former is shown in Figs. \ref{FigE3}, \ref{FigE4}. The "broad'' peaks in the power spectral density plot are indicative of chaos and randomness. 
\begin{figure}[!ht]
					\begin{center} 
\includegraphics[width=0.5\textwidth]{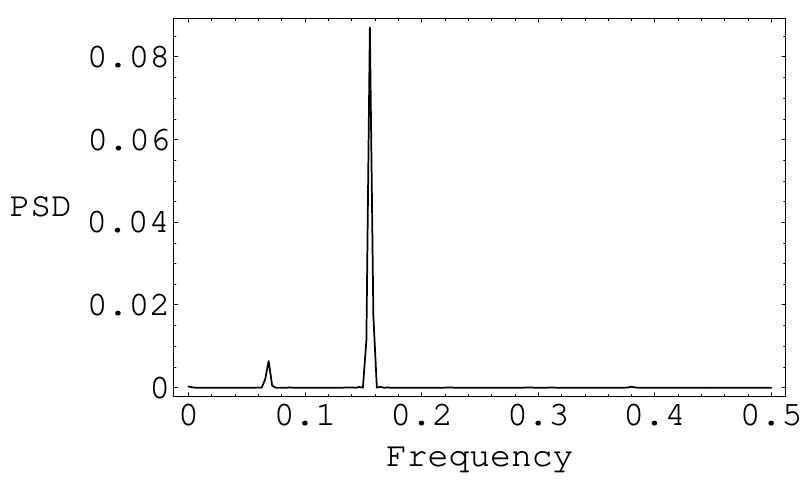}
					\end{center}
				\caption{\small{Power spectral density of System E, $v=-1.1$}}\label{FigE3}
\end{figure}
\begin{figure}[!ht]
					\begin{center} 
\includegraphics[width=0.55\textwidth]{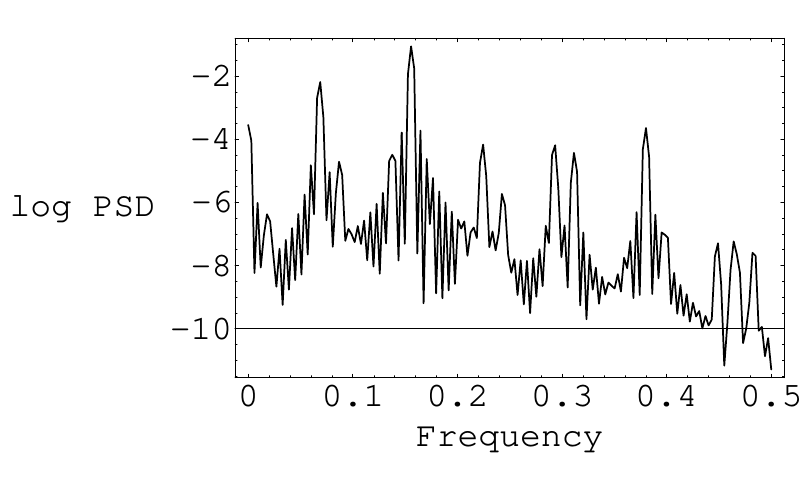}
					\end{center}
				\caption{\small{Log PSD vs. frequency of System E, $v=-1.1$}}\label{FigE4}
\end{figure}

However, we move on to a more quantitative and definitive numerical diagnostic, i.e., the fractal dimension \cite{Choudhury:2}. In order to distinguish low--dimensional (deterministic) chaos from strong randomness, one computes the dimensions as discussed below.  Of several possible alternative definitions \cites{Nayfeh,Grassberger} for the fractal dimensions, we employ the cluster fractal dimension $D$ of Termonia and Alexandrowicz which is defined by
\begin{equation}\label{cl}
n=k[R(n)]^D, n\rightarrow\infty,
\end{equation}
 where $R(n)$ is the average radius of an E--dimensional ball containing $n$ points.  Thus, $D$ is the slope of the plot of $log n$ vs. $log R(n)$. More usefully, if a scaling law \eqref{cl} exists it would show up  as a horizontal line on a plot of $d log n/d log R(n)$ vs. $log n$ with the height of the line being a measure of $D$. Fig \ref{FigE5} shows $D$ which is the height of the approximate horizontal straight line, and we may estimate the converged cluster fractal dimension to be approximately 1.6. This confirms that the System E indeed possesses bounded low dimensional (deterministic) chaotic solutions evolving on a strange attractor with dimension $D$.
\begin{figure}[!ht]
					\begin{center} 
\includegraphics[width=0.5\textwidth]{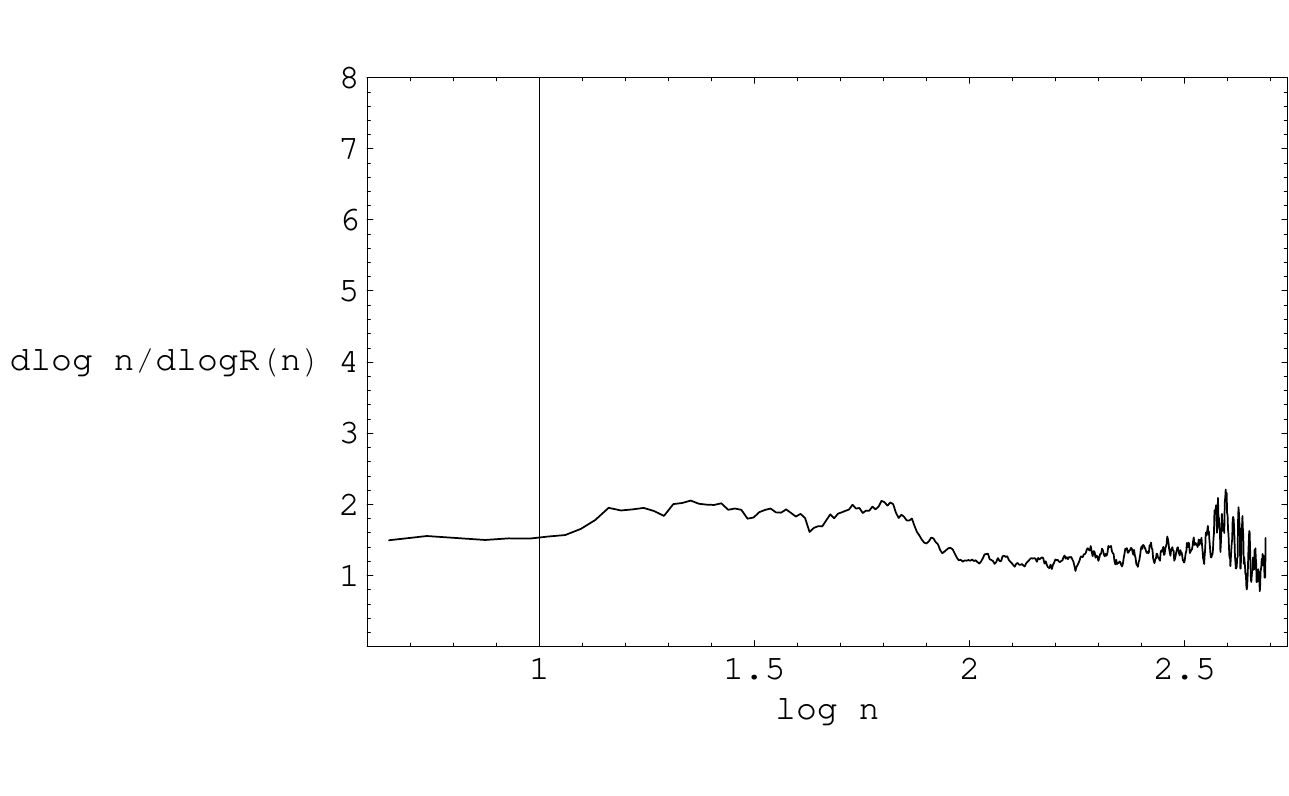}
					\end{center}
				\caption{\small{Cluster dimensions calculation of System E, $v=-1.1$}}\label{FigE5}
\end{figure}

\section{Summary and conclusions}

In this paper traveling wave pattern formation in general reaction--diffusion/predator--prey  models including diffusion in the interspecies interaction terms has been considered. For our first two specific choices of nonlinear terms, the numerical and  mathematical results presented here show either stable equilibrium behavior as in System A, or stable periodic spatial patterns as in System B. Systems C/D exhibit aperiodic spatial behavior (including a finite--time singularity using ODE terminology, or an attractor at infinite in dynamical systems parlance). For System E we also have aperiodic behavior within a diffusive volume, hence the patterns evolve chaotically on a strange attractor.

Various immediate applications of these results suggest themselves. In particular, future work will address specific reaction--diffusion systems such as the Belousov-Zhabotinsky system. Other work in progress includes pulse--train dynamics, as well as the possibility of unsteady pulse solutions in such systems, similar to those recently observed and analyzed in the famous cubic--quintic Ginzburg--Landau equation.

\bibliographystyle{unsrt}
\bibliography {bibliography}

\begin{bibdiv}
\begin{biblist}

\bib{Turing}{article}{
      author={Turing, A.M.},
       title={{The chemical basis of morphogenesis}},
        date={1952},
     journal={Philos. Roy. Soc. B},
      volume={237},
       pages={37},
}

\bib{Wardlaw}{article}{
      author={Wardlaw, C.W.},
       title={{Evidence relating to the diffusion--reaction theory of
  morphogenesis}},
        date={1955},
     journal={New Phytol.},
      volume={54},
       pages={39},
}

\bib{Othmer}{article}{
      author={Othmer, H.G.},
      author={Scriven, L.E.},
       title={{Instability and dynamic pattern in cellular networks}},
        date={1971},
     journal={J. Theoret. Biol.},
      volume={32},
       pages={507},
}

\bib{Segel:1}{article}{
      author={Segel, L.A.},
      author={Jackson, J.L.},
       title={{Dissipative structure: and explanation and an ecological
  example}},
        date={1972},
     journal={J. Theoret. Biol.},
      volume={37},
       pages={545},
}

\bib{Gierer:1}{article}{
      author={Gierer, A.},
      author={Meinhardt, H.},
       title={{A theory of biological pattern formation}},
        date={1972},
     journal={Kybernetic},
      volume={12},
       pages={30},
}

\bib{Gierer:2}{article}{
      author={Gierer, A.},
       title={{Generation of biological patterns and form}},
        date={1981},
     journal={Prog. Biophys. Molec. Biol.},
      volume={27},
       pages={1},
}

\bib{Meinhardt}{book}{
      author={Meinhardt, H.},
       title={{Models of biological pattern formation}},
   publisher={Academic Press},
     address={New York},
        date={1982},
}

\bib{Granero}{article}{
      author={Granero, M.~I.},
      author={Porati, A.},
      author={Zanacca, D.},
       title={{A bifurcation analysis of pattern formation in a diffusion
  governed morphogenetic field}},
        date={1977},
     journal={J. Math. Biol.},
      volume={4},
       pages={21},
}

\bib{Keener}{article}{
      author={Keener, J.P.},
       title={{Activators and inhibitors in pattern formation}},
        date={1978},
     journal={Stud. Appl. Math.},
      volume={59},
       pages={1},
}

\bib{Segel:2}{book}{
      author={Segel, L.A.},
       title={{Taxes in ecology and cell biology}},
   publisher={Springer--Verlag},
     address={Berlin},
        date={1984},
}

\bib{Smoller}{book}{
      author={Smoler, J.},
       title={{Shock waves and reaction--diffusion equations}},
   publisher={Springer--Verlag},
     address={Berlin},
        date={1984},
}

\bib{Rothe}{book}{
      author={Rothe, F.},
       title={{Global solutions of reaction--diffusion equations}},
   publisher={Springer--Verlag},
     address={Berlin},
        date={1984},
}

\bib{Fife}{book}{
      author={Fife, P.C.},
       title={{Mathematical aspects of reacting and diffusing systems}},
   publisher={Springer--Verlag},
     address={New York},
        date={1979},
}

\bib{Murray:1}{article}{
      author={Murray, J.D.},
       title={{A prepattern formation mechanism for animal coat markings}},
        date={1981},
     journal={J. Theoret. Biol.},
      volume={88},
       pages={161},
}

\bib{Mimura}{article}{
      author={Mimura, M.},
      author={Murray, J.D.},
       title={{On a diffusive predator--prey model which exhibits patchiness}},
        date={1978},
     journal={J. Theoret. Biol.},
      volume={75},
       pages={249},
}

\bib{Bard}{article}{
      author={Bard, J.},
       title={{A model for generating aspects of zebra and other mammalian coat
  patterns}},
        date={1981},
     journal={J. Theoret. Biol.},
      volume={93},
       pages={363},
}

\bib{Schaller}{book}{
      author={Schaller, H.C.},
       title={{Neurohormones and their functions in hydra}},
   publisher={Plenum},
     address={London},
        date={1982},
}

\bib{Murray:2}{article}{
      author={Murray, J.D.},
      author={Maini, P.K.},
      author={Tranquillo, R.T.},
       title={{Mechano--chemical models for generating biological pattern and
  form}},
        date={1988},
     journal={Phys. Reports},
      volume={59},
       pages={171},
}

\bib{Levin}{article}{
      author={Levin, S.A.},
      author={Segel, L.A.},
       title={{Pattern generation in space}},
        date={1985},
     journal={SIAM Rev.},
      volume={27},
       pages={45},
}

\bib{Murray:3}{book}{
      author={Murray, J.D.},
       title={{Mathematical Biology}},
   publisher={Springer--Verlag},
     address={Berlin},
        date={1989},
}

\bib{Keshet}{book}{
      author={Edelstein-Keshet, L.},
       title={{Mathematical models in biology}},
   publisher={Random House},
     address={New York},
        date={1988},
}

\bib{MacDonald}{book}{
      author={MacDonald, N.},
       title={{Time lags in biological models}},
   publisher={Springer--Verlag},
     address={Berlin},
        date={1978},
        note={Lecture notes in biomathematics},
}

\bib{Cushing}{book}{
      author={Cushing, J.M.},
       title={{Integrodifferential equations and delay models in population
  dynamics}},
   publisher={Springer--Verlag},
     address={Berlin},
        date={1977},
        note={Lecture notes in biomathematics},
}

\bib{Choudhury:1}{article}{
      author={Choudhury, R.S.},
       title={{On bifurcation and chaos in predator--prey models with delay}},
        date={1992},
     journal={Chaos Solitons and Fractals},
      volume={2},
       pages={393},
}

\bib{Nayfeh}{book}{
      author={Nayfeh, A.H.},
      author={Balachandran, B.},
       title={{Applied nonlinear dynamics}},
   publisher={John Wiley},
        date={1995},
}

\bib{Grassberger}{article}{
      author={Grassberger, P.},
      author={Procaccia, I.},
       title={{Measuring the strangeness of strange attractors}},
        date={1983},
     journal={Physica D},
      volume={9},
       pages={189},
}

\bib{Press}{book}{
      author={Press, W.H.},
      author={Flannery, B.P.},
       title={{Numerical recipes in C}},
   publisher={Cambridge University Press},
     address={Cambridge},
        date={1988},
}

\bib{Choudhury:2}{article}{
      author={Choudhury, R.S.},
      author={Krise, S.},
       title={{Bifurcation and chaos in predator--prey models with delay and a
  laser--diode system with self--sustained pulsations}},
        date={2003},
     journal={Chaos Solitons and Fractals},
      volume={16},
       pages={59},
}

\end{biblist}
\end{bibdiv}
\end{document}